\documentclass[generic,preprint]{imsart}
\RequirePackage{natbib}

\startlocaldefs

\usepackage{amssymb,amsmath,amsfonts,amsthm,amscd,amstext,mathtools}
\usepackage{latexmac}
\usepackage[utf8]{inputenc}
\usepackage{xspace,enumerate}
\usepackage{graphicx}
\usepackage[pdftex]{hyperref}
\usepackage[charter]{mathdesign}
\usepackage[colorinlistoftodos,prependcaption,textsize=tiny]{todonotes}
\presetkeys{todonotes}{color=blue!30}{}
\numberwithin{equation}{section}
\newtheorem{theorem}{Theorem}[section]

\newtheorem{lemma}[theorem]{Lemma}
\newtheorem{proposition}[theorem]{Proposition}

\newtheorem{remark}[theorem]{Remark}
\newtheorem{definition}[theorem]{Definition}

\newtheorem{assumption}{Assumption}

\long\def\xcom#1{}

\newcommand{\dbphi}{\bar{D}_{\varphi}}

\date{\today}

\newcommand{\pac}{path-continuous\xspace}
\newcommand{\pdi}{path-differentiable\xspace}

\endlocaldefs

\begin{document}

\begin{frontmatter}
\title{Some deterministic structured population models which are limit of
  stochastic individual based models}
\runtitle{Structured population model limit of Individual based model}

\begin{abstract} \footnote{\today}
  The aim of this paper is to tackle part of 
  the program set by Diekmann et al. in their seminal paper
  \cite{DiekGHKMT01}. We quote

  ``It remains to investigate whether, and in what sense, the nonlinear
  deterministic model formulation is the limit of a stochastic model
  for initial population size tending to infinity''

  We set a precise and general framework for a stochastic individual
  based model : it is a piecewise deterministic Markov process defined
  on the set of finite measures. We then establish a law of large
  numbers under conditions easy to verify. Finally we show how this
  applies to old and new examples.

\end{abstract}

\author{\fnms{Philippe}
  \snm{Carmona}\corref{}\ead[label=e1]{philippe.carmona@univ-nantes.fr}
\ead[label=u1,url]{http://www.math.sciences.univ-nantes.fr/~carmona/}}

\address{Laboratoire de Math\'ematiques Jean Leray UMR 6629\\
Universit\'e de Nantes, 2 Rue de la Houssini\`ere\\
BP 92208, F-44322 Nantes Cedex 03, France\\ \printead{e1}
}

\runauthor{P. Carmona et al.}



\end{frontmatter}

\newcommand{\spm}{ structured population model \xspace}
\section{Introduction}
We shall focus on linear \spm, as defined in
\cite{DiekGMA98,DiekGHKMT01} : the evolution of individuals ($i$-state
evolution) depends on the individual state, and of the rest of the
population, but not on environmental variables. We shall consider
nonlinear \spm in a forthcoming paper \cite{CAR18b}.\footnote{In fact
  what we call linear model are not linear in the strict sense of \cite{DiekGMA98,DiekGHKMT01} since the evolution may
  depend say on the proportion of individuals of a certain type} We shall
furthermore restrict ourselves to \spm whose building blocks for the
$i$-model are solutions of ODE (see e.g. \cite{AckIto05, DieGylMet07,
  KOOI00}).

We shall establish a law of large number: when initial population goes
to infinity, the measure valued stochastic process converges in
distribution to a deterministic function, measure valued, that
satisfies an integrodifferential equation which is the equation for a
deterministic \spm equation.

There has already been considerable work done in this direction : when
the individuals have discrete traits, for SIR and compartmental models
\cite{Kurtz1970}, when the individuals have continuous traits
\cite{fourniermeleard04, ChaFerMel08,ChaFerMel08b}, for age \spm
\cite{FerTran09,Tran08,Sol87}. The probabilistic toolbox used consists
mainly of a representation in terms of Poisson random measures, and an
identification of martingale problems with subtle limit theorems.

We introduce in this paper a significantly different toolbox. The
stochastic \spm is a PDMP, a \emph{Piecewise Deterministic Markov Process},
on the space of finite measures over $E$, with $E$ the state space for
$i$-model.\footnote{Let us observe  that in former works (see e.g. \cite{TranMetz13}) the
  deterministic evolution of the population between jump times has
  already been made clear, even if they do not explicitly introduce a PDMP}

The first advantage of our framework is that you do not
have to introduce explicit Poisson random measures, nor to perform a
clever labelling of individuals (see
e.g. \cite{fourniermeleard04,TranMetz13}) ; the martingale problem
also is implicit since it is the one associated to the PDMP.

A second advantage of our approach is that we are able to introduce
fairly complex $i$-model evolutions. (This in fact is the starting
point of our work : we wanted to model norovirus evolution in
individuals). An example is given in section
\ref{subsec:hospathexample} where we obtain a new, non trivial,
integro differential equation with transport terms.

Another new feature of our approach is that we can incorporate, e.g. when we have
reproduction, mean numbers of descendants.

\smallskip
Some Notations : \begin{itemize}
  \item $\nu \ll \mu$ means that the
measure $\nu$ is absolutely continuous with respect to $\mu$.
\item $\varphi \sharp m$ is the image of measure $m$ by the measurable
function $\varphi$.
\end{itemize}

\section{The linear stochastic structured population model}

\subsection{Definitions}
The stochastic \spm, abbreviated SSPM, is a PDMP (see
Appendix \ref{sec:appendicepdmp}) on the state $\Mrond_F(E)$ of finite measures over a
measurable  state
space $(E,\Erond)$: $E$ is metrisable, separable and locally compact ;
$\Erond$ is the Borel $\sigma$-field.

\bigskip
The \emph{deterministic dynamic} is driven by a continuous time
dynamical system on $E$ ,  a measurable map
\begin{equation*}
\begin{aligned}
\varphi : &\R^+ \times E &\to & E\\
&(t,x) &\to & \varphi_t(x)
\end{aligned}
\end{equation*}
such that $\varphi_t \circ \varphi_s = \varphi_{t+s}$ and $\varphi_0=id$, and
$t\to \varphi_t(x)$ continuous for any $x$. 

We lift this dynamical system to the space $\Mrond_F(E)$, $\phi : \R^+
\times \Mrond_F(E) \to \Mrond_F(E)$ by the prescription $\phi_t :=
\varphi_t \sharp m$ is the image of the measure $m$ by $\varphi_t$, that is
for every bounded measurable $h$:
$$ \crochet{\phi_t(m), h} := \int h(y) \phi_t(m)(dy) = \int
h(\varphi_t(x))\, m(dx)\,.$$
Then for any $h\in\dbphi$ we have $f(m)=\crochet{m,h} \in \Drond_\phi$
and
$ A_\phi f (m) = \crochet{m,A_\varphi h}$.

\bigskip
The \emph{jump dynamic} is a measurable kernel $\mu:\Mrond_F(E) \to
\Mrond_F(\Mrond_F(E))$. A \emph{population}, a $p$-state in the terminology of \cite{DiekGMA98}, is a
finite measure on $E$ : $m \in \Mrond_F(E)$.

Here is the basic construction step:
\begin{itemize}
\item Given a transition rate function $\alpha:\Mrond_F(E) \times E \to
\R_+$, $\alpha(m,x)$ is the transition rate of the individual $x$ in
the population $m$. 
\item Given a reproduction kernel $k:\Mrond_F(E)
\times E \to \Mrond_S(E)$, so that is $k(m,x)$ is a signed finite
measure on $E$
\item
We define a kernel
\begin{equation}
  \label{eq:conskernelalk}
  \mu(m,dm') := \int_E m(dx) \alpha(m,x) \delta_{m + k(m,x)}(dm')
\end{equation}
\end{itemize}
Of course, since we want that $\mu(m)$ has support on $\Mrond_F(E)$
we require that if the measure $m + k(m,x)$ is not positive, then $\alpha(m,x)=0$.

Usually, to build complex models, we add a finite number of kernels :
given $\alpha_i,k_i$ we let
$$ \mu(m,dm') = \sum_i \mu_i(m,dm')= \int_E m(dx) \sum_i \alpha_i(m,x) \delta_{m +
  k_i(m,x)}(dm')\,.$$
The transition rate of this kernel is
$$ q(m) = \int_{\Mrond_F(E)} \mu(m,dm') = \sum_i \int_E m(dx) \alpha_i(m,x)\,.$$

\begin{definition} The PDMP $\nu$ driven by such $(\mu,\varphi)$ is called a
  $(\mu,\varphi)$ SSMP.
\end{definition}

It is evidently \emph{stable}, that
  is does not explode, in finite time if the rate function $q(m)$ is
  bounded. We shall however use another non explosion criterion more
  suitable for complex models.

 \bigskip
 Eventually, a specific transition rate function, of interaction type, may be built by
 integrating a measurable function mutation kernel
 $$ \alpha(m,x) := \int \bar{\alpha}(m,x,y) \m(dy)\,.$$

 \subsection{Properties} Let us be more precise on the generator of
 this PDMP (see the Appendix).

 \begin{lemma}
   1) Assume that $h:E\to \R$ is \pac for $\varphi$ and bounded. Then the
   function $f:\Mrond_F(E) \to \R$ defined by $f(m)=\crochet{m,h}$ is \pac for $\phi$ and bounded.

   2) Assume furthermore that $h$ is \pdi for $\varphi$ with
   $A_\varphi h$ bounded. Then $f$ is \pdi for $\phi$ with
   $$ A_\phi f(m) = \crochet{m,A_\varphi(h)} = \int m(dx)
   A_\varphi(h)(x)\,.$$
   Therefore, we have the formula
  
 \begin{align*}
   L^\nu f(m) &= A_\phi f(m) + \int \mu(m,dm') (f(m+m') -f(m)) \\
   &= \crochet{m,A_\varphi(h)} + \sum_i \int m(dx) \alpha_i(m,x)
   \crochet{k_i(m,x),h}\,.
 \end{align*}
\end{lemma}
\begin{proof}
  1) Fix $m \in \Mrond_F(E)$. We have:
  $$ f(\phi_t(m)) = \crochet{\phi_t(m),h} = \crochet{\varphi_t\sharp
    m, h} = \int h(\varphi_t(x))\, m(dx)\,.$$
  By assumption, for any $x$, $t\to h(\varphi_t(x))$ is
  continuous. Since it is bounded, we infer by dominated convergence
  that $t\to f(\phi_t(m))$ is continuous.

  2) The proof is similar : we differentiate under the integral sign
  and use $\displaystyle \frac{d}{dt} h(\varphi_t(x)) = A_\varphi h(\varphi_t(x))$.
\end{proof}

\subsection{Examples} \label{subsec:examples}
 Let us see now how versatile this SSMP framework is by obtaining
 classical and non classical  models.
 \subsubsection{The Basic stochastic SIR }\label{subsec:basicsir}
 The state space is made of three compartments $E=\ens{S,I,R}$ with
 $\sigma$-field $\Erond=\Prond(E)$. There is no deterministic
 evolution $\varphi_t=id$ : individual stay in their
 compartment. Therefore the process is just a continuous time Markov
 chain on $\Mrond_F(E)$.
 
 The first kernel  models the prescription ``infected people recover at rate
 $\gamma>0$'':
 $$ \alpha_1(m,x)=\mu_1(x,dz) = \gamma \un{x=I}\,,\qquad 
 k_1(m,x) = \delta_R - \delta_I\,.$$

 At rate $\gamma$ an infected individual recovers, that is it is
 removed from the population, the term $-\delta_I$, and a recovered is
 added, the term $+\delta_R$.
 
 The second kernel is an interaction kernel, that models the prescription
 ``susceptible are infected by infected people at per capita rate
 $\beta>0$ ``:
 $$ \bar{\alpha}_2(m,x,y) = \beta \un{x=S,y=I}\,, \quad k_2(m,x) =
 \delta_I - \delta_S\,.$$
which yields:
$$ \alpha_2(m,x)= \beta I(m) \un{x=S}\,, $$
with $I(m)=m(\ens{I})$ the number of infected people. The total rate
function is, with $S(m)=m(\ens{S})$,
$$ q(m) =q_1(m)+q_2(m) = \gamma I(m) + \beta I(m) S(m)\,.$$

 \subsubsection{Age since infection structured SIR}\label{subsec:agessir}
This model has also been introduced by Kermack and McKendrick, see
\cite{kermac27}, \cite[section 1.5.2]{perthamebook07}, \cite[section
13.2]{martchevabook15}.

The infection rate and recovery rate depend on the age since infection
, i.e. are two measurable functions $\gamma,\beta:\R_+ \to \R_+$. The
state space is
$$ E=\ens{S} \cup \ens{I}\times [0,+\infty[ \cup \ens{R}\,,\quad
\Erond=\sigma(\ens{S},\ens{R},\ens{I}\times B, B \in \Brond(\R))\,.$$

We let $c:E\to \ens{S,I,R}$ be the compartment function and $a(I,t)=t$
be the age function. Then the recovery mechanism is described by

$$ \alpha_1(m,x)= \gamma(a(x)) \un{c(x)=I} \,,\quad
k_1(m,x)=\delta_R -\delta_x\,.$$
And the infection mechanism is induced by
$$\bar{\alpha}_2(m,x,y) = \beta(a(y)) \un{c(x)=S,c(y)=I}
\,, k_2(m,x) = \delta_{(I,0)}(dz)- \delta_x$$
so that
$$\alpha_2(m,x) = \lambda(m) \un{c(x)=S}\,,$$
with $\lambda(m)$ the total rate of infection of population $m$:
$$\lambda(m) := \int \beta(a(y)) \un{c(y)=I} m(dy)\,.$$
The rate functions are
$$ q_1(m) = \int \gamma(a(y)) \un{c(y)=I} m(dy)\,,\quad q_2(m) = S(m)
\lambda(m)\,$$
and the driving  dynamical system is
$$ \varphi_t(S)=S, \varphi_t(R)=R,\varphi_t(I,s) = (I, t+s)\,.$$

\subsubsection{A simple host/pathogen interaction with immigration of
  pathogen}\label{subsec:hospath}

This example cannot be considered as an age structured epidemiological
model, and therefore is totally new.

The host pathogen interaction model comes from Gilchrist and Sasaki,
see e.g. \cite[section 14.2.2, equation (14.3)]{martchevabook15} and
the references therein.

The continuous dynamical system $(\varphi_t)_{t\ge 0}$ is the flow of the
ODE
\begin{equation}\label{eq:gilsasa}
     \left\{
     \begin{split}
       \frac{dP}{dt} &= r P -c BP\\
       \frac{dB}{dt} &= a BP\,.
     \end{split}
\right.
\end{equation}
with state space $E=(0,+\infty)^2$. $B$ is the quantity of  immune
cells, $P$ of pathogen cells ; $r>0$ is the pathogen reproduction
rate, $c>0$ the pathogen clearance rate by the immune cells, $a$ the
stimulation of immune cells production by the pathogen. The pathogen
is always cleared, since every trajectory converges to $(P^*=0,B^*)$
with $B^*$ depending on the initial conditions.

We shall see that if we impose some absolute continuity on a mutation
kernel, then the large population limit is an integro differential
\emph{transport equation}.

\subsubsection{The Bell-Anderson model}
This model is described in \cite[Section 3]{MetzDiek86} and
\cite{Jag99}. A cell with size $x$ dies with intensity $d(x)$ and
splits into two cells of equal size $x/2$ with intensity
$b(x)$. An individual cell growths with rate $g(x)>0$.

There is a minimal size $a>0$ and maximal cell size $4a$. Initially
all cells have size in $[a,2a]$, no cell with size smaller than 2a can
divide.

The state space is therefore $E=\etc{a,4a}$ with its Borel sigma
field. Cell-growth is modelled by the flow $(\varphi_t)_{t\ge 0}$ of
the ODE
$$ \dot{x} = g(x)\,.$$
We assume thus that $E$ is stable by the flow.

Death is modelled by
$$ \alpha_1(m,x)= d(x) \ge 0\,,\quad k_1(m,x) = -\delta_x\,,$$
and reproduction by
$$ \alpha_2(m,x) = b(x)\ge 0\,,\quad k_2(m,x)= -\delta_x + 2
\delta_{x/2}\,.$$
We assume that $b(x)=0$ is $x< 2a$.

The total rate of population $m$ is thus
$$ q(m) = q_1(m) + q_2(m) = \int m(dx) (b(x) + d(x))\,.$$

\section{Large Population Limit : Law of Large Numbers}

Let $(\nu_t)_{t\ge 0}$ be a $(\mu,\varphi)$ SSMP process. Our first concern
is to ensure that the size of the population
$\nu_t(E)=\crochet{\nu_t,1}$ does not explodes in
finite time. This shall not only ensure non explosion but also yield
useful bounds on the size.

\begin{assumption}[growth control] \label{ass:growth}
  \begin{enumerate}
    \item The variation norms of the reproduction kernels $k_i$ are
uniformly  bounded : $C_k:=\sup_{i,m,x} \norme{k_i(m,x)}_{VT} < +\infty$.
\item Let $\bar{I}$ be the set of  $i$ such that there exists $m,x$ with
  $\alpha_i(m,x)>0$ and  $k_i(m,x,E)>0$. Then
  there exists a constant $C_q$ such that
  $$ q_i(m)  \le C_q (1 + \crochet{m,1}) \qquad(i\in\bar{I}, m\in\Mrond_F(E))\,.$$
\end{enumerate}
\end{assumption}
\begin{remark}
  The first assumption ensures that a jump cannot increase or decrease
  the population of more than $C_k$ unit. The second assumption
  controls the rate of the jumps.
\end{remark}
\begin{proposition}[mass control]\label{pro:masscontrol}
  Assume that for some $p\ge 1$ we have
  $$ \esp{\crochet{\nu_0,1}^p} < +\infty\,.$$
  Then, under the Assumption \ref{ass:growth}, 
  \begin{enumerate}
    \item The $(\mu,\phi)$-SSMP
  process is defined on $[0,+\infty[$.
  \item There exists a constant $C>0$ such that for any $t>0$,
$$  \esp{\crochet{\nu_t,1}^p} \le C e^{Ct}\,.$$
  \end{enumerate}
\end{proposition}

When the size of the initial population is approximately $n$,
e.g. if $\nu_0=\sum_{i=1}^n \delta_{x_i}$ is deterministic, the $x_i$
being the individuals, we can renormalize by considering $(\unsur{n}
\nu_t)_{t\ge 0}$. The following theorem yields sufficient conditions
for a law of large number, i.e. the convergence of the renormalization
to a deterministic limit.

\begin{assumption}[regular kernel]\label{ass:regular}
  Assume that the kernel $\mu$ satisfy  on the
  space $\Mrond_S(E)$ of bounded variation signed measure on $E$, that
  for any $m\in \Mrond_S(E)$ and any $r>0$, there exists constants
  $C_1,C_2$ such that with $B(m,r) = \ens{m' \in \Mrond_S(E) :
    \norme{m-m'}_{TV} \le r}$,
  \begin{gather}
    \sup_{i,x\in E, m'\in B(m,r)} \norme{\alpha_i(m,x)k_i(m,x) -\alpha_i(m',x)k_i(m',x)}_{TV}
    \le C_1 \norme{m-m'}_{TV} \\
    \sup_{i,m'\in B(m,r)x\in E} \alpha_i(m',x) \le C_2.
  \end{gather}
\end{assumption}
\begin{remark}
  This assumption entails some uniform Lipschitz bound in the total
  variation norm that is necessary not only to  establish uniqueness in the
  limiting integro differential equation (Proposition
  \ref{pro:uniqueness}) but also to prove tightness (compactness) of
  processes ad thus the existence of limits (see Step 1 of the proof
  of Theorem \ref{thm:llngsir}).
\end{remark}
\begin{theorem}\label{thm:llngsir}
  Let $(\nu^n_t)_t\ge 0$ be a $(\mu^{(n)},\varphi)$ SSMP process
  with the scaling
  $$ \alpha_i^{(n)}(nm,x)=\alpha_i(m,x)\,,\quad k_i^{(n)}(nm,x)= k_i(m,x)\,,$$
  with $\mu$ a fixed jump dynamic satisfying
  Assumptions \ref{ass:growth},\ref{ass:regular} and the bound on the total rate function
  $$ q(m) \le C'_q(1 + \crochet{m,1} + \crochet{m,1}^2)\,.$$
  Let $X^n_t = \unsur{n} \nu^n_t$ be the renormalized measure valued
  process. Assume that
  \begin{enumerate}
    \item For some $p\ge 3$, $\sup_n \esp{\crochet{X^n_0,1}^p} < +\infty$.
    \item There exists $\xi_0 \in \Mrond_F(E)$ such that $X^n_0$
    converges to $\xi_0$ in probability.
    \item  The set $\dbphi$ of functions $h$ bounded, \pdi, such that
    $A_\varphi h$ is bounded and \pac is dense in $C_0(E)$.
  \end{enumerate}
Then $(X^n_t)_{t\ge 0}$ converges in probability in
$\mathbb{D}([0,T],\Mrond_F(E))$ to a deterministic continuous function
$(\xi_t)_{0\le t\le T}$ which satisfies: for all $h:\R_+\times E \to
\R$, such that for fixed $x$, $t\to h(t,x)$ is $C^1$ and for every
$t$, $x\to h(t,x)\in \dbphi$,
\begin{equation}\label{eq:integrodifxit}
  \frac{d}{dt}\crochet{\xi_t,h} = \crochet{\xi_t,\partial_th(t,.) +
    A_\varphi(h)(t,.)} +\sum_i  \int
  \xi_t(dx) \alpha_i(\xi_t,x) \crochet{h(t,.),k_i(\xi_t,x,.)}\,. 
  \end{equation}

\end{theorem}
\begin{remark}
  Of course, the scaling assumption may be replaced by a convergence
  assumption such as
  $$ \lim_{n\to +\infty} \alpha_i^{(n)}(nm,x)=\alpha_i(m,x)\,,\quad
  \lim_{n\to +\infty} k_i^{(n)}(nm,x)= k_i(m,x)\,.$$
  We have encountered no need for such a generality in our studies,
  and thus we leave the generalisation to an interested reader.

\end{remark}

\section{Large Population Limit : the associated PDE}
\newcommand{\barx}{\bar{\Xrond}}

In order to show that equation \eqref{eq:integrodifxit} may be seen as
a PDE, we need some absolute continuity assumption. We shall need more
structure on the state space (what we require looks a lot like the
framework used in \cite{Benleb15}).

The state space $E$ is the union $E= E_1 \times E_2$, with $E_1$ a
finite set (of compartments) and $E_2 = \cup_{i\in I_2} \ens{i}
\times \barx_i$ the finite union of compartments with a
continuous trait : $\barx_i$ is the closure of an open connected set
$\Xrond_i$ of $\R^{d_i}$.

Let $(F^i)_{i\in I_2}$ be smooth vector fields $F^i: \barx_i \to
\R^{d_i} $ and $\varphi^i=(\varphi^i_t)_{t\in \R}$ be  the flow induced by
them : $t \to \varphi^i_t(x)$ is the solution of
$$ \dot{x} = F^i(x)$$
with initial condition $x(0)=x$.

The dynamical system on $E$ is then $\varphi_t(x) = x$ if $x \in E_1$ and
$\varphi_t(i,x)= \varphi^i_t(x) $ for $ (i,x) \in E_2$. The domain
$\dbphi$ contains bounded functions $h : E\to \R$ such that for any
$i\in E_2$, $x\to h(i,x) \in C^1(\barx_i)$, and 
 $ x \to \nabla h(i,x)  F^i(x)$
bounded. For these functions 

$$A_\varphi h(x)=0 \text{  if }x \in E_1\quad\text{ and } A_\varphi (i,x) = \nabla
h(i,x)  F^i(x)\text{ for }(i,x) \in E_2\,.$$

 Hence $\dbphi$ is dense in $C_0(E)$ since it contains
constant functions and functions $h$ s.t. for any $i\in I_2$, $x \to
h(i,x) \in C^\infty_K(\Xrond_i)$ (functions $C^\infty$ with compact support).

The reference measure on $E$ will be
$$ \lambda = \sum_{x \in E_1} \delta_x + \sum_{i\in I_2} \delta_i
\otimes Leb(\barx_i)$$

It should be clear then that for any $t$,the  image $\varphi_t\sharp \lambda$ is
      absolutely continuous with respect to $\lambda$:
      $$ \varphi_t\sharp \lambda \ll \lambda \qquad(\forall t\in
      \R)\,.$$

      Indeed this is true on every compartment in $E_1$ and in $\ens{i}
      \times \barx_i$ we just need to use the jacobian of the flow.
Furthermore, the dual of $A_\varphi$ in the sense of distributions
is $$A_\varphi^*(m) = - \sum_{i\in I_2} div_i(m(i,x) F^i(x))$$
with $div_i$ the divergence in the sense of distribution on each $\Xrond_i$.

\begin{theorem}\label{thm:edptrans}
  Let $(\xi_t)_{i\in\etc{0,T}}$ be a solution of
  \eqref{eq:integrodifxit} with $(\mu,\varphi)$ satisfying the
  Assumptions \ref{ass:growth}, \ref{ass:regular}.
  
  Assume furthermore that
  \begin{itemize}
    \item $\xi_0 \ll \lambda$.
    \item for every $i,m,x$ $k_i(m,x)$ is absolutely continuous with
      respect to $\lambda$ with density say $k_i(m,x,z)$ a measurable
      positive function defined on $\Mrond_F(E)\times E\times E$.
  \end{itemize}
 
  Then for any $t\in[0,T]$, $\xi_t \ll \lambda$ and the density
  $\xi(t,x)$ satisfy in the weak sense the PDE
  \begin{align*}
    \partial_t \xi(t,x) - A_\varphi^*(\xi_t)(x) = -
    \int \xi(t,x') \sum_i \alpha_i(\xi_t,x') k_i(\xi_t,x',x) \lambda(dx')\,.
   \end{align*}
\end{theorem}

\section{Proofs of large population limits}

\subsection{Proof of Proposition \ref{pro:masscontrol}}
\begin{proof}
  We shall follow closely the proof of \cite[Theorem
  3.1]{fourniermeleard04}.
  Given $a>0$ we let $\tau_a:=\inf\ens{t\ge 0 : \crochet{\nu_t,1} \ge
    a}$ and $f(m)=\crochet{m,1}^p$. Observe that
  $$A_\phi f(m) = p \crochet{m,1}^{p-1} \crochet{m,A_\varphi(1)} =0\,.$$
Therefore,
  \begin{align*}
    0 \le L^\nu f(m) &= \sum_i \int m(dx) \alpha_i(m,x)
    \etp{(\crochet{m+k_i(m,x),1})^p - \crochet{m,1}^p} \\
    &= \sum_{i \in \bar{I}} \int m(dx) \alpha_i(m,x)
    \etp{(\crochet{m+k_i(m,x),1})^p - \crochet{m,1}^p} \\
    &\le \sum_{i \in \bar{I}} \int m(dx) \alpha_i(m,x) C_p \crochet{m,1}^{p-1}\,,
  \end{align*}
  where $C_p>0$ satisfies
  $$ (x+C_k)^p - x^p \le C_p (1 + x^{p-1}) \qquad(x\ge 0)\,.$$
  Hence,
  \begin{align*}
    0 \le L^\nu f(m) &\le C_p \sum_{i\in \bar{I}} \crochet{m,1}^{p-1}
    q_i(m)\\  &\le   C_p C_q \sum_{i\in \bar{I}} \crochet{m,1}^{p-1}
    (1+\crochet{m,1})\\
    &\le C (1 + \crochet{m,1}^p)\,.
  \end{align*}
  Therefore the martingale $M^f_t$ in the decomposition
  \eqref{eq:semidecpdmp} is such that $M^f_{t\wedge \tau_a}$ is a true
  martingale, and
  \begin{align*}
    \esp{f(\nu_{t\wedge \tau_a})} &= \esp{f(\nu_0)} +
    \esp{\int_0^{t\wedge \tau_a} L^\nu f(\nu_s)\, ds } \\
    &\le \esp{\crochet{\nu_0,1}^p} + C' \int_0^t
    \esp{1+\crochet{\nu_{s\wedge \tau_a},1}^p}\, ds\,.
  \end{align*}
  Gronwall's Lemma ensures then the existence of a constant $C$, not
  depending on $a$, such
  that for every $t>0$
  \begin{equation}\label{eq:gronmaspuisp}
    \esp{1+\crochet{\nu_{t\wedge \tau_a},1}^p} \le C e^{Ct}\,.
  \end{equation}
  This implies that $\lim_{a\to +\infty}\tau_a=+\infty$ almost surely, that is
  non explosion. Then, taking limits in \eqref{eq:gronmaspuisp},
  yields the desired upper bound.

\end{proof}
\xcom{Compléments sur la preuve: si on n'a pas $\tau_\infty = \lim
  \tau_a = +\infty $ ps , alors il existe $t_0$ tel que
  $\prob{\tau_\nfty < t_0} >0$ et alors
  \begin{multline*}\esp{\crochet{\nu_{t_0 \wedge \tau_a}, 1} } \\
    \ge \esp{\crochet{\nu_{t_0 \wedge \tau_a}, 1} \un{\tau_\infty <
        t_0}}\\
    \ge  \esp{\crochet{\nu_{\tau_a}, 1} \un{\tau_\infty <
        t_0}} \ge (a-C_k) \prob{\tau_\infty < t_0} 
  \end{multline*}
  et on fait $a \to +\infty$ pour obtenir une contradiction
  }
  \subsection{Study of equation \eqref{eq:integrodifxit}}
\begin{proposition}[Uniqueness] \label{pro:uniqueness}Assume that  the kernel
  $\mu$ satisfy Assumption \eqref{ass:regular} and that bounded \pac
  functions are dense in bounded functions. Given $m_0 \in \Mrond_F(E)$, there is
  at most only one solution  $(\xi_t)_{t\in[0,T]}$  of
  \eqref{eq:integrodifxit} that satisfies $\xi_0=m_0$.
\end{proposition}
\begin{proof}
Assume that $\xi,\xi'$ are two solutions with the same initial value
$\xi_0=\xi_0'=m_0$. Fix $t>0$. Let $g:E \to \R$ measurable bounded
\pac for $\varphi$, and let
$h(s,x)=g(\varphi_{t-s}(x))$. Then, by definition of the generator
$A_\varphi$ of the dynamical system $\varphi$, $h$ is a solution of
$$ \partial_s h + A_\varphi(h) = 0\,,\quad h(t,x) = g(x)\,.$$
Therefore, injecting this into equation \eqref{eq:integrodifxit} yields
\begin{equation}\label{eq:xitk}  \crochet{\xi_t,g} =  \crochet{m_0,h(0,.)} + \sum_i \intot
  F_i(s,\xi_s)\, ds
\end{equation}

with
$$ F_i(s,m) = \int m(dx) \alpha_i(m,x) \crochet{k_i(m,x),h(s,.)}\,.$$

By continuity of $t\to \xi_t$ and $t\to \xi'_t$, we can chose $r>0$
such that for all $t\in\etc{0,T}$, $\xi_t$ and $\xi'_t$ are in
$B(m_0,r)$.

Observe that  Assumptions \eqref{ass:growth} and \eqref{ass:regular} implies that for
$m,m'$ in $B(m_0,r)$
\begin{align}
  \valabs{F_i(s,m) -F_i(s,m')} &\le \int  m(dx) \valabs{\int
    (\alpha_i(m,x) k_i(m,x,dz)
    -\alpha_i(m',x) k_i(m',x,dz)) h(s,z)} \notag\\
  &+ \int (m(dx)-m'(dx)) \alpha_i(m',x) \valabs{\int h(s,z)k_i(m',x,dz)} \notag\\
  &\le  \norme{g}_\infty \int  m(dx)
  \norme{\alpha_i(m,x) k_i(m,x)-\alpha_i(m',x)k_i(m',x)}_{TV} \notag\\  &+\norme{g}_\infty
  \norme{m-m'}_{TV} \sup_x \alpha_i(m',x)
  \norme{k_i(m',x)}_{TV}\notag\\
  &\le \norme{g}_\infty  \norme{m-m'}_{TV} (C_1 \crochet{m,1} +
  C_kC_2) \notag\\
  & \le C \norme{g}_\infty  \norme{m-m'}_{TV} \,. \label{eq:lipFM}
\end{align}
 We  obtain that for a constant
 $C$ that does not depend on $g$ nor $T$, for all $t\in\etc{0,T}$,
 $$  \valabs{\crochet{\xi_t,g}- \crochet{\xi'_t,g}} \le C
 \norme{g}_\infty \int_0^t \norme{\xi_s -\xi'_s}_{TV}\, ds\,.$$

 Since $g$ is arbitrary,  this implies that
 $$ \norme{\xi_t -\xi'_t}_{TV} \le C \int_0^t \norme{\xi_s
   -\xi'_s}_{TV}\, ds$$
 and we conclude by Gronwall's Lemma that for all $t\in \etc{0,T}$,
 $ \xi_t=\xi'_t$.
\end{proof}
\subsection{Proof of Theorem \ref{thm:llngsir}}
 We shall follow closely the lines (and the arguments) of the proof of
  \cite[Theorem 5.3]{fourniermeleard04}. Since we have already
  established uniqueness of the limit equation in
  Proposition~\ref{pro:uniqueness}, we shall establish 
  \begin{description}
    \item[Step 1] Tightness of the family of distributions of $X^n$ in
    $\mathbb{D}(\etc{0,T}, (\Mrond_F(E),v))$ (that is when
    $\Mrond_F(E)$ has the vague topology).
    \item[Step 2] Show any limit in distribution satisfies the
    limiting equation
    \item[Step 3] Convergence in distribution in
    $\mathbb{D}(\etc{0,T}, (\Mrond_F(E),w))$ (with the topology of
    weak convergence) to the unique solution of the limiting equation.
  \end{description}

  \subsubsection*{Step 1} Since $\dbphi$ is dense in $C_0(E)$ the set
  of continuous functions with a limit at infinity, according to
  \cite{roelly86}, it is enough to prove that for any $h\in \dbphi$
the process $\crochet{X^n,h}$ is tight in
$\mathbb{D}(\etc{0,T},\R)$. We shall show first that
$$ \sup_n \esp{\sup_{t\in \etc{0,T}} \valabs{\crochet{X^n,h}}} <
+\infty$$
This is a direct consequence of the boundedness of $h$ and of
Lemma~\ref{lem:unifmasscontrol}.

Hence, according to Aldous criterion \cite{aldous89} and the Rebolledo
criterion of \cite[Corollary 2.3.3]{joffemetivier86}, it suffices to   prove the
tightness of the  martingale part and the drift part of
$\crochet{X^n,h}$. More precisely,  we only need to prove that for
the decomposition of $f(X^n_t)=\crochet{X^n_t,h}$:
$$  f(X^n_t) = f(X^n_0) + M^{f,n}_t + U^{f,n}_t$$
we have for a constant $C_{h,T}$ depending only on $h$ and $T>0$, that
for every stopping times $ 0 \le S \le S' \le S+\delta \le T$ , and
every $n\ge 1$,
we have
\begin{align*}
  \esp{\valabs{U^{f,n}_{S'} -  U^{f,n}_{S}}} &\le C_{h,T} \delta\,,\\
   \esp{\valabs{\crochet{M^{f,n}}_{S'} - \crochet{M^{f,n}}_{S}}} &\le C_{h,T} \delta\,.
 \end{align*}
 Observe first that $U^{f,n}_t = \int_0^t L^{X^n} f(X^n_s)\, ds$ and
 that, thanks to scaling,
 \begin{align*}
   L^{X_n} f(m) &= L^{\nu_n} f(\unsur{n }.) (nm) \\
   &= \crochet{m, A_\varphi(h)} + \int n m(dx) \sum_i
   \alpha_i^{(n)}(nm,x) \crochet{k_i^{(n)}(nm,x), \unsur{n}h} \\
   &= \crochet{m, A_\varphi(h)} + \int  m(dx) \sum_i
   \alpha_i(m,x) \crochet{k_i(m,x), h}\,.
 \end{align*}
 Therefore,
 \begin{align*}
   \valabs{ L^{X_n} f(m)} &\le \crochet{m,1} \norme{A_\varphi(h)}_\infty
   + \norme{h}_{\infty} \sum_i \int m(dx) \alpha_i(m,x) \norme{k_i(m,x)}_{TV} \\
   &\le  \crochet{m,1} \norme{A_\varphi(h)}_\infty+ \norme{h}_{\infty}
   C_k q(m) \\
   &\le  \crochet{m,1} \norme{A_\varphi(h)}_\infty+ \norme{h}_{\infty}
   C'_q(1 + \crochet{m,1} + \crochet{m,1}^2) \\
    &\le C_h (1 + \crochet{m,1}^2)\,.
\end{align*}
Consequently, by Lemma~\ref{lem:unifmasscontrol}, since $p\ge 3$,
\begin{align*}
  \esp{\valabs{U^{f,n}_{S'} -  U^{f,n}_{S}}}&\le C_h \delta \esp{
    \sup_{s\le T} (1+\crochet{X^n_s,1}^2)}  \le C_{h,T} \delta\,.
\end{align*}
Similarly,
$$ \crochet{M^{f,n}_t } = \int_0^t (L^{X^n} f^2 - 2 f L^{X^n}
f)(X^n_s)\, ds\,,$$
and similarly  we have the bound
\begin{align*}
  (L^{X^n} f^2 - 2 f L^{X^n} f)(m) &=  \int  n\,m(dx) \sum_i
   \alpha_i(m,x) \crochet{k_i(m,x), \unsur{n}h}^2\\
  &\le \unsur{n} C_k^2 \norme{h}_{\infty}^2 q(m) \\
  &\le \unsur{n} C_h (1 + \crochet{m,1}^2)\,.
\end{align*}
Therefore we obtain,
\begin{align*}
  \esp{\valabs{\crochet{M^{f,n}}_{S'} - \crochet{M^{f,n}}_{S}}}&\le
  \delta C_h \esp{\sup_{s\le T} (1+\crochet{X^n_s,1}^2)}  \le C_{h,T} \delta\,.
\end{align*}
\subsection*{Step 2}
Let $(X_t)_{t\in\etc{0,T}}$ be the limit in distribution in
  $\mathbb{D}(\etc{0,T}, (\Mrond_F(E),v))$ of a subsequence
  $X^{\kappa(n)}$. By construction, almost surely,
  $$ \sup_{t\in\etc{0,T}} \sup_{h\in L^\infty(E),\norme{h}_{\infty}\le
  1} \valabs{ \crochet{X^n_t,h} - \crochet{X^n_{t-},h}} \le
\frac{2}{n}\,.$$

Therefore $X$ is almost surely strongly continuous. Let $h:\R_+\times E \to
\R$, such that for fixed $x$, $t\to h(t,x)$ is $C^1$ and for every
$t$, $x\to h(x,t)\in \dbphi$.

 For any measured valued function
$m\in C([0,T],\Mrond_F(E))$ we let
\begin{align*}
 \psi(m) &= \crochet{m_t,h} -\crochet{m,h} - \int_0^t ds
\crochet{m_s,\partial_s h(s,.) + A_\varphi(h)(s,.)} \\
&- \int_0^t\, ds\, \int m_s(dx) \sum_i \alpha_i(m_s,x) \crochet{k_i(m_s,x,.),h(s,.)} \,.
\end{align*}
We are going to show that 
$$\esp{\valabs{\psi(X)}}=0\,.$$
Observe that for $f(s,m)=\crochet{m,h(s,.)}$ we have 
$$ M^{f,n}_t = \psi(X^n)\,,$$
where we have the semi martingale decomposition
$$ f(t,X^n_t) = f(0,X_0) + M^{f,n}_t + U^{f,n}_t$$
with $$ U^{f,n}_t = \intot (\partial_s f(s,X^n_s) +
L^{X^n}f(s,.)(X^n_s))\, ds\,.$$

The preceding computations applied to stopping times $S=0$
and $S'=t$ yield
$$\esp{\psi(X^n)^2}=\esp{(M^{f,n}_t)^2}= \esp{\crochet{ M^{f,n}}_t} = \int_0^t \etp{ L^{X_n}(f^2(s,.)) - 2 f
L^{X_n}(f(s,.))}(X^n_s)\,ds \le C_{h,T} \unsur{n}  \to 0\,.$$

We can prove as in the proof of Proposition \ref{pro:uniqueness} (see
inequality \eqref{eq:lipFM}) that
thanks to Assumption 2, $\psi$ is Lipschitz in the total variation norm:
$$ \valabs{\psi(m) -\psi(m')} \le C (\norme{A_\varphi(h)}_\infty +
\norme{\partial_s h}_\infty +
\norme{h}_\infty) \norme{m-m'}_{TV}\,.$$
Since $X$ is a.s. strongly continuous, this implies that $\psi$ is
a.s. continuous at $X$. Since $\psi(X^n)$ is bounded in $L^2$ it is
Uniformly Integrable, and we have
 $$ \esp{\valabs{\psi(X)}} = \lim_{n\to \infty}
\esp{\valabs{\psi(X^{\kappa(n)})}} = \lim_{n\to \infty}
\esp{\valabs{M^{f,\kappa(n)}_t}} = 0\,.$$
Therefore, $X=\xi$ the unique solution of equation \eqref{eq:integrodifxit}.

\subsection*{Step 3} The previous steps imply that $X^n$ converges in
distribution in $\mathbb{D}(\etc{0,T}, (\Mrond_F(E),v))$ to $(\xi_t)_{t\in[0,T]}$ the unique solution of
equation \eqref{eq:integrodifxit}.

If we apply Step 2 to the function $h=1 \in \dbphi$, we
obtain that $(\crochet{X_t^n,1})_{0\le t\le T}$ converges in
distribution to $(\crochet{\xi_t,1})_{0\le t\le T}$. Since this
limiting process is continuous,  a criterion proved in
\cite{meleardroelly93}  implies that this convergence holds in $\mathbb{D}(\etc{0,T}, (\Mrond_F(E),w))$.

\subsection{Proof of Theorem \ref{thm:edptrans}}
 Let $g=1_A$ with $\lambda(A)=0$ and  let
$h(s,x)=k(\varphi_{t-s}(x))$. 

Then for every $t\ge 0$, $k(\varphi_t(x))\ge 0$ and since
$\varphi_t\sharp\lambda \ll \lambda$, we have
$$ \int g(\varphi_t(x)) d\lambda(x) = \int g(y) d
\varphi_t\sharp\lambda(y) =0\,.$$
Therefore,$ g(\varphi_t(x))=0$, for $\lambda$ almost every $x$ and this
implies that $h(s,x) =0$ for $\lambda$ almost every $x$.

Let us examine  equation
\eqref{eq:xitk}. We obtain, since $\xi_0 \ll \lambda$ and $k_i(m,x)
\ll \lambda$,
\begin{align*}
  0 \le \crochet{\xi_t,g} &\le \crochet{\xi_0,g\circ \varphi_t}\\
  &+
\int_0^t  ds \sum_i \int \xi_s(dx) \alpha_i(\xi_s,x) \int_E
k_i(\xi_s,x,dz) h(s,z) =0\,.
\end{align*}

Therefore   $\xi_t \ll \lambda$. We let $\xi_t(dx) =
\xi(t,x) \lambda(dx)$.

Now given a function $h:E\to \R$, such that for every $i\in I_2$, $x\to
h(i,x) \in C^\infty_K(\Xrond_i)$, we manipulate  \eqref{eq:integrodifxit}
and use the adjoint operator to get
\begin{align*}
  \frac{d}{dt} \int \xi(t,x) h(x) d\lambda(x)  &= \int A_\varphi^*
  (\xi(t,.))(x) h(x) d\lambda(x) 
  \\
  &+ \sum_i \int d\lambda(x) \xi(t,x) \alpha_i(\xi_t,x) \int k_i(\xi_t,x,z) h(z)\,
  d\lambda(z) \\
  &= \int A_\varphi^*
  (\xi(t,.))(x) h(x) d\lambda(x) 
  \\
  &+  \int d\lambda(z) h(z) ) \int \sum_i  \xi(t,x) \alpha_i(\xi_t,x) \int k_i(\xi_t,x,z)   d\lambda(x)\,,
\end{align*}
which is exactly the desired weak sense PDE.

\section{Applications and Examples}

In this section we review the examples introduced in section
\ref{subsec:examples} and show how to verify the assumptions of Theorem~\ref{thm:llngsir}.
\subsection{Basic SIR (continuation of
\ref{subsec:basicsir})} 

The rate function is

$$ q(m) = \gamma I(m) + \beta I(m) S(m) \le C(\crochet{m,1} +
\crochet{m,1}^2\,.$$

Therefore Assumption \ref{ass:growth} is satisfied. Since
$I(m)=m(\ens{I})$ and $S(m)=m(\ens{S})$, Assumption \ref{ass:regular}
is also satisfied with $C_k=2$.

The scaling relation  requires that
$\mu^{(n)}$ is associated to $\gamma_n=\gamma$ and
$\beta_n=\frac\beta{n}$. We consider $X^n_t =\unsur{n} \nu^n_t$ where
$\nu_n$ is the SSMP driven by $\mu^{(n)}$. We shall just assume that $X^n_0=\unsur{n} \nu^n_0$ converges in
probability to $\xi_0$.

Therefore the
law of large numbers yields that  the limiting deterministic process
$(\xi_t)_{t\ge 0}$
satisfies for every $h$
\begin{align*}
  \frac{d}{dt} \crochet{\xi_t,h} &= \int \xi_t(dx)  (\gamma
  \un{x=I} (h(R) -h(I)) + \beta I(\xi_t) \un{x=S} (h(I) -h(S)) \\
  &=
 \gamma\, I(\xi_t)\, (h(R) -h(I)) + \beta\, I(\xi_t)\, S(\xi_t)\, (h(I) -h(S))\,.
\end{align*}

Therefore if $S(t)=S(\xi_t)$, $I(t)=I(\xi_t)$ and
$R(t)=R(\xi_t)$ taking $h(x)=\un{x=a} $ for $a\in \ens{S,I,R}$ yields
that $(S,I,R)$ satisfy the system

 \begin{equation}
   \label{eq:sysbasicsir}
   \left\{
     \begin{split}
       \frac{dS}{dt} &= -\beta S I\\
       \frac{dI}{dt} &= \beta S I - \gamma I\\
       \frac{dR}{dt} &= \gamma I\,.
     \end{split}
\right.
 \end{equation}

This is the classical system of ODE introduced by \cite{kermac27}.

\subsection{Age structured SIR (Continuation of Example
\ref{subsec:agessir})} \label{subsec:agessirplus}  We assume that the functions $\gamma$ and $\beta$
are bounded so that the total rate is quadratic at most:
$$ q(m) = q_1(m) + q_2(m) \le \norme{\gamma}_\infty \crochet{m,1} +
\norme{\beta}_\infty \crochet{m,1}^2\,.$$
The scaling relation  impose that
$\mu^{(n)}$ is associated to the functions $\gamma_n(a)=\gamma(a)$ and
$\beta_n(a) =\unsur{n} \beta(a)$.

The dynamical system has generator $A_\varphi h(R)=A_\varphi h(S) =0$
and $A_\varphi h(I,s) = \frac{d}{ds} h(I,s)$ and
$$\dbphi=\ens{h \text
  { bounded } : t\to h(I,t) \in C^1_b}\,.$$

The limiting process $(\xi_t)_{t\ge 0}$ satisfies for $h\in \dbphi$
\begin{align*}
  \frac{d}{dt} \crochet{\xi_t,h} &= \int \xi_t(dx) \un{c(x)=I}
  \frac{d}{dt} h(I,t) \\
  &+ \int  \xi_t(dx) \gamma(a(x)) \un{c(x)=I} (h(R) -h(x)) \\
  &+ \int  \xi_t(dx) \lambda(\xi_t)   \un{c(x)=S} (h(I,0) -h(x))\,.
\end{align*}
Therefore, if $S=S(t) =\xi_t({S})$ and $R(t)=\xi_t(\ens{R})$, taking
$h(x)=\un{c(x)=S}$ yields
$$ \frac{dS}{dt} = -\lambda(\xi_t) S_t$$
and with $h(x) = \un{c(x)=R}$, we get
$$ \frac{dR}{dt} = \int \gamma(a(x)) \un{c(x)=I} \xi_t(dx)\,.$$

If $h(x)= \un{c(x)=I} g(a(x))$ and if $\kappa_t$ is the image of the
restriction of $\xi_t$ to $\ens{i}\times[0,+\infty[$ by the function
$a(x)$, that is
$$ \int \xi_t(dx) \un{c(x)=I} f(a(x)) = \int_{\R_+} \kappa_t(da)
f(a)\,,$$
then we get
$$ \frac{d}{dt}\int \kappa_t(da) g(a) = \int g'(a) \kappa_t(da) + g(0)
S_t \lambda(\xi_t)$$

Assuming the absolute continuity of the initial conditions with
respect to Lebesgue measure
$\kappa_0(da) = i(0,a)\, da$, then we obtain, following the proof of
Theorem\ref{thm:edptrans}, that $\kappa_t(da) = i(a,t)\, da$ and therefore,
since $i(a,t)=0$ for $a<0$, 
$$ \frac{d}{dt}\int g(a) i(a,t)\, da = \int_0^{+\infty} g'(a) i(t,a)\,
da + g(0) S_t \lambda(\xi_t)$$
with
$$ \lambda(\xi_t) = \intof \beta(a) i(t,a)\,da\,.$$
A simple integration by parts, for $g \in C^1$ with support in a
compact $[0,M]$ proves that in a weak sens $i(t,a)$ is solution of
$$ \partial_t i(t,a) + \partial_a i(t,a) = 0\,,\qquad i(t,0) = S_t
\lambda(\xi_t)\,.$$
This is exactly the PDE system with boundary conditions derived by
Kermack and McKendrick, see \cite[section 1.5.2]{perthamebook07} or
\cite[section 13.2]{martchevabook15}.

\subsection{A simple host/pathogen interaction (continuation of \ref{subsec:hospath})}\label{subsec:hospathexample}
 We assume that at a
constant rate $\gamma>0$ an individual $x=(b,p)$ mutates to and
individual $x'=(b',p')$ with a density $\psi_x(x')$ with respect to
Lebesgue measure on $E=(0,+\infty)^2$. This mutation may just be an
injection of a random quantity of pathogens and a destruction of a
random quantity of immune cells. The  kernel is defined through
$$ \alpha(m,x)= \gamma\,,\quad k(m,x,dz) = \psi_x(z)\, dz$$
with $dz$ the Lebesgue measure. It has  constant rate function $q(m) = \gamma$. Then $\mu^{(n)}$ is also
   associated to $\gamma$ and $\psi$. Then, Theorem \ref{thm:edptrans}
   implies that the limit process has a density $\xi(t,x)$ that
   satisfy in a weak sense the PDE:
   $$ \partial_t \xi(t,x) + div(\xi(t,x) F(x)) = -\gamma \xi(t,x ) +
   \int_{(0,+\infty)^2} \psi_{x'}(x) \xi(t,x')\, dx'\,,$$
   with $F$ the smooth vector field
   $$F(x)F(p,b) =
   \begin{pmatrix}
     rp -bcp \\ abp
   \end{pmatrix}\,.$$

\subsection{The Bell Anderson model}
We assume boundedness of the rates so $q(m) \le C \crochet{m,1}$ and
Assumption 1 is satisfied. The regularity Assumption 2,  is then also
satisfied since the  rates do not depend on the population $m$ and
the total variation of $k_i$ is bounded. Scaling is also trivially
satisfied so that we obtain the limit equation
\begin{equation}\label{eq:mesbellanderson} \frac{d}{dt} \crochet{\xi_t,h} = \int_E \xi_t(dx) \etp{g(x) h'(x) -d(x)
  h(x) + b(x) (-h(x) + 2h(x/2))}\,.
\end{equation}
This is exactly the weak form of the PDE \cite[section
3.4]{MetzDiek86}
$$ \frac{\partial n}{\partial t} (t,x) + \frac{\partial}{\partial x}
(g(x) n(t,x)) = -d(x) n(t,x) - \beta(x) n(t,x) + 4 \beta(2x)
n(t,2x)\,.$$

Of course one may object that this PDE may only be inferred from
\eqref{eq:mesbellanderson} if we prove that $\xi_t(dx)$ has a density
$n(t,x)$ with respect to Lebesgue measure on $E$. With some additional
assumptions on $g$ this is however possible to establish. For example,
if one assumes that $g$  satisfies
$$2 g(x) =
g(2x)\,.$$
Comparison of ODE solutions yield that if
$h(s,x)=\gamma(\phi_{t-s}(x))$ we have $h(s,x) = 2 h(s,x/2)$ and thus
if $\xi_0 \ll \lambda$, and $\gamma(x)=1_A(x)$ with $\lambda(A)=0$,  we have, as in the proof of Theorem
\ref{thm:edptrans}
\begin{align*}
  0 \le \crochet{\xi_t,\gamma} \le \crochet{\xi_0,g\circ \varphi_t} -
  \intot ds \int \xi_s(dx) d(x) h(s,x) \le \crochet{\xi_0,g\circ
    \varphi_t} =0\,.
\end{align*}
And therefore $\xi_t \ll \lambda$.

\medskip
Another way to obtain rigorously a limit PDE is to change the
splitting mechanism as proposed in \cite[section 3.2]{MetzDiek86}:
$$k_2(m,x,dy) = -\delta_x(dy) + 2 \pi(x,y)\, dy$$
with $\pi(x,y)=\pi(x,x-y)$ and thus $\int \pi(x,y) y \, dy = x/2$.
   
\appendix
\section{Definitions and basic properties of PDMP} \label{sec:appendicepdmp}
A PDMP (Piecewise Deterministic Markov Process) is a Markov process
when its randomness is only given by a jump mechanism : between the
jump times the trajectories are deterministic (see
e.g. \cite{davispdmp84,davisbook93,jacobsenbook06})
Let us collect here some facts and results on PDMP  from the literature, mainly from
\cite{jacobsenbook06}.

The state space is a measurable space $(G,\Grond)$ (usually a Borel space
or a Polish space). We are given three ingredients:
\begin{itemize}
\item a \emph{rate function} $q: G \to \R_+$ measurable.
\item a \emph{probability transition kernel} on $G$, that is a
  measurable function $r:G \to \Prond(G)$ the set of probabilities on $G$
  endowed with the weak convergence topology. We shall write 
$ r(x,C) = r_x(C)$ for $C\in \Grond$ and also write $r(x,dx')$ for the
probability measure $r_x$. We assume $r(x,\ens{x})=0$. Sometimes
these two ingredients are joined by considering a \emph{kernel}
$\mu(x,C) := q(x) r(x,C)$, that is a measurable map $\mu : G \to
\Mrond_F(G)$ into the space of finite measures over $G$.
\item A \emph{continuous time dynamical system} on $G$, 
 that is a map
\begin{equation*}
\begin{aligned}
\phi : &\R^+ \times G &\to & G\\
&(t,x) &\to & \phi_t(x)
\end{aligned}
\end{equation*}
such that $\phi_t \circ \phi_s = \phi_{t+s}$ and $\phi_0=id$. We
assume that for each $x$, $t\to \phi_t(x)$ is continuous.
\end{itemize}

Given $x_0 \in G$ we construct a two sequences $(T_n){n\ge 0}$ and
$(Y_n)_{n\ge 0}$ by specifying the conditional laws:
\begin{itemize}
\item $T_0=0$ and $Y_0=x_0$.
\item The law of $T_{n+1}$ given $(T_k,Y_k)=(t_k,y_k), {0\le k\le n}$
  is given by
$$ \prob{T_{n+1} > t + t_n \mid (T_k,Y_k)=(t_k,y_k), {0\le k\le n}} =
\exp\etp{- \int_0^t q(\phi_s(y_n))\, ds}\,.$$
\item The law of $Y_{n+1}$ given $(T_k,Y_k)=(t_k,y_k), {0\le k\le n}$
  and $T_{n+1}=t_{n+1}$  is given by
$$ \prob{Y_{n+1} \in C\mid (T_k,Y_k)=(t_k,y_k), {0\le k\le
    n},T_{n+1}=t_{n+1}} = r(\phi_{t_{n+1}-t_n}(y_n),C)\,.$$
\end{itemize}
We assume \emph{stability} that is $T_n \to +\infty $ a.s. This is the
case if the 
rate is bounded : for all $y\in G$, $q(y) \le \bar{q}<
+\infty$. Indeed, then we have stochastic domination $\tau_i \prec
T_{i+1} - T_i$ where $\tau_i$ is IID exponential of parameter
$\bar{q}$.

Eventually we let
\begin{equation} X_t = \phi_{t-T_n}(Y_n) \,\quad\text{for}\quad T_n \le t <
T_{n+1}\,.
\end{equation}
Then  $(X_t)_{t\ge 0}$ is a strong homogeneous Markov process with respect to its
natural filtration $\Frond^X_t$ (see \cite[Theorem
7.2.1]{jacobsenbook06}), called a PDMP of parameter $(q,r,\phi)$
starting from $x_0$. We let $\PP_{x_0}$ denote the law of $X$.

Observe that if $\phi$ is the constant flow, then $(X_t)_{t\ge 0}$ is
an homogeneous continuous time Markov chain on $G$ with transition
kernel $p(x,dx') = q(x) r(x,dx')$.

Observe also that if the rate function $q$ is constant along the flow
of the dynamic system, that is $q(\phi_t(x))=q(x)$, then the
construction of the sequence may be simplified as (if $\Erond(a)$ denotes
the law of an exponential random variable of parameter $a$)
\begin{gather}
 {\Lrond}(T_{n+1}-T_n \mid T_1,Y_1, \ldots, T_n,Y_n) =
 {\Erond}(q(Y_n))\\
\Lrond(Y_{n+1} \mid T_1,Y_1, \ldots, T_n,Y_n,T_{n+1}) = r(\phi_{T_{n+1}-T_n}(Y_n),.)
\end{gather}

We also have an Itô formula and the definition of an associated
infinitesimal generator (see \cite[Theorem
7.6.1]{jacobsenbook06}).

A measurable function $h : G\to \R$ is \emph{path-continuous}
(resp. \emph{path-differentiable}) is for all $x$ the function
$$ t \to h(\phi_t(x))$$
is continuous (resp. differentiable). If $h$ is continuous it is of
course \pac. If it is \pdi, we define
$$ A_\phi h (y) := \lim_{s \to 0^+} \unsur{s} \etp{h(\phi_s(y)) -h
  (y)}\,,$$
and see that
$$ \frac{d}{dt} h(\phi_t(y)) = A_\phi h (\phi_t(y))\,.$$
If $t\to \phi_t(y)$ is differentiable, $h$ is $C^1$ and
$$ a(y) := \lim_{s \to 0^+} \unsur{s} \etp{\phi_s(y) -
  y}\,,$$
then $A_\phi h = \nabla h . a$.

We assume that for any bounded \pac function $h$ the function
\begin{equation}\label{eq:contqr}
   t \to q(\phi_t(y)) \int r(\phi_t(y),dz) h(z)
\end{equation}

is continuous for any $y$.

The \emph{full infinitesimal generator} for the PDMP is the linear
operator $L$ given by \eqref{eq:defgenpdmp} acting on the domain
$\Drond(L)$ of bounded measurable functions $h:G \to \R$ such that $h$
is \pdi and $A_\phi h$ is \pac, and the function $L h : G \to \R$
given by

\begin{equation}\label{eq:defgenpdmp}
  Lh(x) = A_\phi h(x) + q(x) \int r(x,dx') (h(x') - h(x))\,,
\end{equation}
is bounded (this function is then \pac thanks to~\eqref{eq:contqr}).

We say that a function $f:[0,+\infty[\times G \to \R$ is \emph{bounded on
finite time intervals} if it is bounded on all sets $\etc{0,t}\times
G$.

If $f$ is measurable, bounded on finite time intervals with $t\to
f(t,x)$ continuously differentiable for all $x$ and $x \to f(t,x)$
\pdi for all $t$, then
the process $f(t,X_t)$ has the decomposition
\begin{equation}\label{eq:semidecpdmp} 
  f(t,X_t) = f(0,X_0) + M^f_t + U^f_t
\end{equation}
where $M^f_t$ is a local martingale reduced by the sequence
$(T_n)_{n\ge 1}$, and $U^f_t$ is the continuous predictable process
\begin{equation}
  U^f_t = \int_0^t \Lrond f(s,X_s) \, ds\,,
\end{equation}
with
\begin{equation}
  \Lrond f(t,x) = \partial_t f(t,x) + L(f(t,.))(x)\,.
\end{equation}

If, in addition, the function $\Lrond f$ is bounded, then $M^f$ is a true
martingale and
\begin{equation}
  \esperance{x_0}{f(t,X_t)} = f(0,x_0) + \int_0^t
  \esperance{x_0}{\Lrond f(s,X_s)}\, ds\,.
\end{equation}

Furthermore, if $f^2$ satisfies the same assumptions,  then the predictable quadratic
variation of the local martingale $M^f$ is 
\begin{equation}
  \crochet{M^t}_t = \int_0^t (\Lrond (f^2) - 2f \Lrond f)(s,X_s)\, ds\,.
\end{equation}
A straightforward computation yields that $A_\phi(f^2)(x)= 2 f(x)
A_\phi f(x)$
and therefore
\begin{equation}\label{eq:quadvarpdmp}
 (\Lrond (f^2) - 2f \Lrond f)(t,x) =  q(x) \int r(x,dx')\, \etp{f(t,x') -f(t,x)}^2\,.
\end{equation}

We shall use the decomposition \eqref{eq:semidecpdmp} when $f$ is not
bounded on finite intervals, but $f(t,X_t)$ is locally bounded and
thus write $Lf(x)$ (resp. $\Lrond f(t,x)$) even when $f$ is not in the
domain of the generator.
\section{Additional Lemmas and Propositions}

\begin{lemma}[Uniform mass control]\label{lem:unifmasscontrol}
  Let $(\nu_t)_{t\ge 0}$ be a $(\mu,\varphi)$ SSMP satisfying
  assumption \ref{ass:growth}.
  Assume that for some $p\ge 1$,
  $$ \esp{\crochet{\nu_0,1}^p} < +\infty.$$
  Then for any $q \in \etc{1,\frac{p+1}{2}}$, and any $T>0$
    $$ \esp{ \sup_{ t\le T} \crochet{\nu_t,1}^q} < +\infty\,.$$
  \end{lemma}
  \begin{proof}
    The proof of Proposition \ref{pro:masscontrol} yields that with
    $f(m)=\crochet{m,1}^q$, we have
    $$ f(\nu_t) \le  f(\nu_0) + C' \int_0^t  (1+f(\nu_s))\, ds +
    M^f_t\,.$$
    Therefore, if $Y_t = \sup_{s\le t} f(\nu_s)$, we have for any
    $t\le T$
    $$ Y_t \le Y_0 + \sup_{t \le T} M^f_t + C' t + C' \int_0^t Y_s\,
    ds\,.$$
    We shall be able to apply Gronwall's Lemma, once we show that
    $\esp{\sup_{t \le T} M^f_t} < +\infty$. We shall use
    Cauchy-Schwarz inequality and the quadratic variation process.

    Indeed,
    \begin{align*}
      L f^2(m) -f(m) L f(m) &= \sum_i \int m(dx) \alpha_i(m,x)
      \etp{\crochet{m+k_i(m,x,.),1}^q - \crochet{m,1}^q}^2 \\
      &= \sum_{i \in \bar{I}} \int m(dx) \alpha_i(m,x)
      \etp{\crochet{m+k_i(m,x,.),1}^q - \crochet{m,1}^q}^2 \\
      &\le C (1 + \crochet{m,1}^{q-1})^2 \int \sum_{i\in \bar{I}}
      m(dx) \alpha_i(m,x) \\
      &\le C (1 + \crochet{m,1})(1 + \crochet{m,1}^{q-1})^2 \\
      &\le C ( 1 + \crochet{m,1}^{2q-1}) \le C' (1 +\crochet{m,1}^p)  \,.
    \end{align*}
    Therefore, by Doob's $L^2$ maximal inequality
    \begin{align*}
      \esp{\sup_{t \le T} M^f_t}^2 &\le C
      \esp{\crochet{M^f}_T} \\
      &\le C \int_0^T \esp{(L f^2 - 2 f Lf)(\nu_s)}\, ds\\
      &\le C \int_0^T (1 + \esp{\crochet{\nu_s,1}^p})\, ds < +\infty\,,
    \end{align*}
    where in the last bound we used Proposition \ref{pro:masscontrol}.
  \end{proof}

\bibliographystyle{imsart-nameyear}
  \bibliography{bibsir}

\begin{thebibliography}{27}
\bibitem[\protect\citeauthoryear{Carmona}{}]{CAR18b}
\begin{bunpublished}[author]
\bauthor{\bsnm{Carmona},~\bfnm{Philippe}\binits{P.}}
\btitle{Some non linear deterministic structured population models which are
  limit of stochastic individual based models}.
\bnote{In preparation}.
\end{bunpublished}
\endbibitem
\end{thebibliography}

\end{document}